\documentclass[12pt,leqno]{amsart}

\usepackage[T1]{fontenc}

\usepackage[
  a4paper,
  top=1.2in,
  bottom=1.2in,
  left=1in,
  right=1in,
  headheight=0.5in,
  footskip=0.5in
]{geometry}

\usepackage[none]{hyphenat}  

\usepackage{amssymb}   
\usepackage{mathrsfs}  
\usepackage{accents}   
\usepackage{mathdots}  
\usepackage{nccmath}   
\usepackage{stmaryrd}  
\usepackage{marvosym}  
\usepackage{scalerel}
\usepackage{xurl}

\usepackage{graphicx}  
\usepackage[dvipsnames]{xcolor}
\usepackage{tikz}     
\usepackage{float}    

\usepackage[format=plain,justification=centering]{caption}
\usepackage{subfig}   

\usepackage[shortlabels]{enumitem}
\setlist[itemize,enumerate]{leftmargin=11mm}

\usepackage{comment}

\usepackage[foot]{amsaddr}
\usepackage[nodayofweek]{datetime}

\usepackage[thmtools-compat]{keytheorems}

\usepackage{cite}

\usepackage{hyperref}
\hypersetup{
  colorlinks=true,
  linkcolor=black,
  urlcolor=Blue,
  citecolor=black
}

\usepackage[capitalise,nameinlink]{cleveref}

\crefname{conjecture}{Conjecture}{Conjectures}
\crefname{question}{Question}{Questions}
\crefname{observation}{Observation}{Observations}

\newkeytheorem{theorem}[parent=section]

\newkeytheorem{
  statement,lemma,proposition,corollary,conjecture,
  question,observation,claim,fact
}[sibling=theorem,style=plain]

\newkeytheorem{axiom,definition,example}[
  sibling=theorem,
  style=definition
]

\newcommand{\mf}{\mathfrak}
\newcommand{\mca}{\mathcal}
\newcommand{\poi}{\mathbb{N}}

\newcommand{\bt}{\mathsf{B}}
\newcommand{\smallparen}[2][0.7]{\scaleobj{#1}{(}#2\scaleobj{#1}{)}}

\newcommand{\bigmath}[2][0.75]{%
  \mathord{\scalebox{#1}{$\displaystyle #2$}}%
}

\newcounter{tbox}
\crefname{tbox}{Statement}{Statements}

\newcommand{\sta}[1]{%
  \par\medskip\medskip
  \refstepcounter{tbox}%
  \noindent
  \parbox{\linewidth}{(\thetbox) \emph{#1}}%
  \par\vspace*{0.3cm}%
}

\newcommand{\mylongtitle}[1]{%
  \ifodd\value{page}%
    \protect\parbox{0.97\linewidth}{#1}\hfill
  \else
    \hfill\protect\parbox{0.97\linewidth}{#1}%
  \fi
}

\title{Off-diagonal Ramsey numbers of 3-graphs}

\author{Lior Gishboliner$^{\ast\dagger}$}
\author{Sepehr Hajebi$^{\ast\ddagger}$}
\author{Xiangyu Li$^{\ast}$}

\thanks{$^{\ast}$ Department of Mathematics, University of Toronto,
  Toronto, Ontario, Canada.}
\thanks{$^{\dagger}$ Research supported in part by an NSERC Discovery Grant.}
\thanks{$^{\ddagger}$ Supported by the Canada Postdoctoral Research
  Award (CPRA) program.}

\date{\today}

\begin{document}

\maketitle
\begin{abstract}
A $3$-graph (meaning a $3$-uniform hypergraph) $H$ is \textit{iteratively $3$-partite} if $H$ is isomorphic to a subgraph of some lexicographic power of a single hyperedge. We prove a conjecture of Conlon, Fox, Gunby, He, Mubayi, Suk, Verstra\"{e}te, and Yu, that for a $3$-graph $H$, the off-diagonal Ramsey number $R(H,n)$ is polynomial in $n$ if and only if $H$ is iteratively $3$-partite. As an immediate corollary, we also obtain a resolution of a 2009 problem of Bohman, Frieze and Mubayi, asking for a characterization of all $3$-graphs $H$ for which the minimum number of hyperedges in an $H$-free $3$-graph with chromatic number greater than $k$ is $k^{3+o(1)}$ as $k\to\infty$. We prove a $3$-graph $H$ has the latter property if and only if $H$ is not iteratively $3$-partite.
\end{abstract}

\section{Introduction}\label{sec:introduction}
\subsection{Background and the main result} In this paper, we denote the set of all positive integers by $\poi$, and the set of all positive integers at most $k\in \mathbb Z$ by $[k]$ (so $[k]=\varnothing$ if $k\leq 0$). Hypergraphs have finite vertex sets and nonempty hyperedges. A hypergraph $G$ is denoted by the pair $(V(G),E(G))$ where $V(G)$ is the vertex set and $E(G)$ is the hyperedge set of $G$. A \textit{subgraph} of a hypergraph $G$ is a hypergraph $G'$ with $V(G')\subseteq V(G)$ and $E(G')\subseteq E(G)$. For $X\subseteq V(G)$, we denote by $G[X]$ the subgraph of $G$ with vertex set $X$ and hyperedge set $\{e\in E(G): e\subseteq X\}$. A \textit{copy} of a hypergraph $H$ in $G$ is a subgraph of $G$ isomorphic to $H$.  

For $k\in\poi$, a hypergraph is \textit{$k$-uniform}, also called a $k$-graph, if each of its hyperedges has cardinality $k$. A $k$-graph is \textit{complete} if all $k$-subsets of its vertices are hyperedges, and the complete $3$-graph on $n\in \poi$ vertices is denoted by $K_n^{\smallparen{3}}$.  Let $k,n\geq 2$ be integers and let $H$ be a $k$-graph. The (off-diagonal) \textit{Ramsey number} $R(H,n)$ is the smallest $N\in\poi$ for which every red/blue coloring of the hyperedges of the $N$-vertex complete $k$-graph contains either a red copy of $H$ or a blue copy of the $n$-vertex complete $k$-graph. For $k=2$ (i.e. for graphs), it is well known that $R(H,n)$ is polynomial in $n$ for every fixed $H$. For $k=3$, however, this already fails for $H=K_{4}^{\smallparen{3}}$ (see, for instance, \cite{CFS}). The natural question is therefore which $3$-graphs $H$ retain a polynomial bound in $n$ for $R(H,n)$. This question was first posed explicitly by Conlon at an AIM workshop in 2015 \cite{Conlon2015,mainconjecture}, although its origins go back to work of Erd\H{o}s and Hajnal \cite{ErdosHajnal1972} in 1972.

In this paper, we give a full characterization of all such $3$-graphs. For an integer $t\in \poi$, define the $3$-graph $\bt_t$ recursively as follows: let $\bt_1$ be $K_{3}^{\smallparen{3}}$, and for $t\geq 2$, construct $\bt_t$ by taking three disjoint copies of $\bt_{t-1}$ and adding every hyperedge with one vertex in each copy\footnote{Equivalently, $\bt_t$ for each $t\in \poi$ is the $t^{\rm th}$ lexicographic power of $K_{3}^{\smallparen{3}}$.}. A $3$-graph $H$ is called \textit{iteratively tripartite} if $H$ is isomorphic to a subgraph of $\bt_t$ for some $t\in \poi$.

Erd\H{o}s and Hajnal \cite[Theorem~7]{ErdosHajnal1972} proved that $R(H,n)$ is polynomial in $n$ whenever $H$ is iteratively tripartite (see also \cite{newbounds} and \cite{FoxHe} for related work). Conlon, Fox, Gunby, He, Mubayi, Suk, Verstra\"ete, and Yu \cite{mainconjecture} conjectured that the converse is also true:

\begin{conjecture}[Conlon, Fox, Gunby, He, Mubayi, Suk, Verstra\"ete, Yu {\cite{mainconjecture}}]\label{conj:main}
For a $3$-graph $H$, there exists $c\in\poi$ such that $R(H,n)\leq n^c$ for every integer $n\geq 2$ (if and) only if $H$ is iteratively tripartite.
\end{conjecture}

We prove \Cref{conj:main}. Our main result, \Cref{thm:main} below, is stated without Ramsey numbers. A \textit{stable set} in a hypergraph $G$ is a set $S$ of vertices that contains no hyperedge, and $\alpha(G)$ is the maximum cardinality of a stable set in $G$. For a $3$-graph $H$ and an integer $n\geq 2$, the Ramsey number $R(H,n)$ may thus be defined as the minimum $N\in \poi$ such that every $N$-vertex $3$-graph with $\alpha(G)<n$ contains a copy of $H$.

\begin{restatable}{theorem}{main}\label{thm:main}
For all $c,h\in\poi$ with $h\geq 2$, and every integer $n\geq 3^{{(30h(c+1))}^{6h}}$, there exists a $3$-graph $G$ such that $|V(G)|=n^c$, $\alpha(G)<n$, and every subgraph of $G$ with $h$ hyperedges is iteratively tripartite.
\end{restatable}

There has been considerable work on \Cref{conj:main} preceding our resolution of it in this paper. To mention a few:

\begin{itemize}[itemsep=1mm]
    \item As potential troublemakers for \Cref{conj:main}, \textit{linear} $3$-graphs (those in which every two distinct hyperedges intersect in at most one vertex) were also expected to have polynomial off-diagonal Ramsey numbers. Conlon, Fox, Gunby, He, Mubayi, Suk, and Verstra\"ete \cite{onoffdiagonal} refuted that prediction. Further constructions in higher uniformities were later discovered by He, Nie, Wigderson, and Yu \cite{linear}.

\item A \textit{tight path} in a $3$-graph is a sequence of hyperedges in which consecutive ones intersect in two vertices, and a \textit{tight component} is a maximal set of hyperedges any two of which can be joined by a tight path. In \cite{mainconjecture}, the authors proved \Cref{conj:main} for $3$-graphs with at most two tight components.

\item Very recently, Conlon, He, Nie, and Post \cite{newbounds} proved the conjecture for $3$-graphs with at most five tight components. They also showed that the full conjecture reduces to the case of linear $3$-graphs. (So the linear case, once expected to supply further polynomial examples, turned out to be where the bulk of the difficulty remained.)
\end{itemize}

Our proof of \Cref{thm:main} offers a glimpse into the ``lexicographic amplification'' method developed by the second author in a forthcoming work \cite{Lexamp}. In particular, \Cref{thm:main} follows immediately from the simplest application of that method. However, in its full generality, lexicographic amplification has a much wider scope and brings with it a blizzard of technical work that goes far beyond what is needed to prove \Cref{thm:main}. Given the longstanding interest in \Cref{conj:main}, we believe it is worth giving a concise, self-contained proof which the curious reader can follow in one sitting, and that is the purpose of the present work.

\subsection{Corollaries} From \Cref{thm:main}, we deduce the following explicit superpolynomial lower bound (which we did not put any effort into optimizing):

\begin{corollary}\label{cor:quantlower}
Let $H$ be a $3$-graph that is not iteratively tripartite and let $h=|E(H)|$. For every integer $n\geq 3^{(120h)^{6h}}$, we have
$$R(H,n)>3^{\frac{1}{60h}(\log_3 n)^{1+1/(6h)}}.$$
\end{corollary}

\begin{proof}
Note that $h\geq 2$. Let $L=\log_3 n$, let $K=L^{1/(6h)}/(30h)$, and let $c=\lfloor K\rfloor-1$.
Then
$n=3^L\geq 3^{(30h(c+1))^{6h}}$. By \Cref{thm:main}, there is a $3$-graph on $n^c$ vertices with no copy of $H$ and no stable set of cardinality $n$. Moreover, note that $c=\lfloor K\rfloor-1>K-2\geq K/2$ because $K\geq 4$. Hence
$$R(H,n)>n^c=3^{c\log_3n}> 3^{(K\log_3n)/2}= 3^{\frac{1}{60h}L^{1+1/(6h)}}.$$
\end{proof}

We also derive from \Cref{thm:main} a complete solution to a 2009 problem of Bohman, Frieze and Mubayi, concerning the minimum number of hyperedges in a $3$-graph of large chromatic number and with no copy of a fixed $3$-graph $H$. 

The \textit{chromatic number} of a $3$-graph $G$, denoted $\chi(G)$, is the minimum $k\in \poi$ for which $V(G)$ can be partitioned into $k$ stable sets. For $k\in \poi$ and a $3$-graph $H$, let $m_k(H)$ be the minimum number of hyperedges in a $3$-graph $G$ with $\chi(G)>k$ and no copy of $H$ (and let $m_k(H)=\infty$ if no such $3$-graph exists).  In \cite[Section~7]{bohmanetal}, Bohman, Frieze and Mubayi asked for a complete characterization of all $3$-graphs $H$ for which $m_k(H)=k^{3+o(1)}$ as $k\to\infty$. We prove that:

\begin{corollary}\label{cor:chi}
Let $H$ be a $3$-graph. Then $m_k(H)=k^{3+o(1)}$ as $k\to\infty$ if and only if $H$ is not iteratively tripartite.
\end{corollary}

\begin{proof}
First, assume that $H$ is iteratively tripartite. By \cite[Theorem~7]{ErdosHajnal1972}, there exist $b,B\in \poi$ with $b\geq 2$ such that $R(H,n)\leq Bn^{b}$ for every integer $n\geq 2$. Thus, by the lower-bound part of \cite[Theorem~5]{EHHfree}, we have $m_k(H)\geq B'k^{3+1/(b-1)}$ for all $k\in \poi$, where $B'>0$ is a constant depending on $B$ and $b$. This rules out $m_k(H)=k^{3+o(1)}$ as $k\to\infty$.

Now, suppose that $H$ is not iteratively tripartite. Let $h=|E(H)|\geq 2$ (note that every $3$-graph with at most one hyperedge is iteratively tripartite).  Fix an integer $c\geq 2$, let 
$$k_0=3^{(30h(c+1))^{6h}(c-1)},$$ 
and let $k\in \poi$ with $k\geq k_0$. Let $n=\lfloor 2k^{1/(c-1)}\rfloor$. Then, $n> k^{1/(c-1)}\geq 3^{(30h(c+1))^{6h}}$ (because $k^{1/(c-1)}>1$), and so by \Cref{thm:main}, there is a $3$-graph $G$ with $n^c$ vertices and $\alpha(G)<n$ such that every subgraph of $G$ with $h$ hyperedges is iteratively tripartite. In particular, $G$ contains no copy of $H$. Moreover, $\chi(G)\geq |V(G)|/\alpha(G)>n^{c-1}>k$.
Consequently,
\[
m_k(H)\leq |E(G)|\leq n^{3c}\leq \left(2k^{1/(c-1)}\right)^{3c}\leq \left(k^{\frac{1+(1/c)}{c-1}}\right)^{3c}= k^{3+(6/(c-1))};
\]
where the last inequality follows from $k^{1/(c(c-1))}\geq 3^{(30h(c+1))^{6h-1}}>2$.

On the other hand, by \cite[Proposition~2]{alonhighchi}, for all $k,r\geq 2$, the number of hyperedges in every $r$-graph of chromatic number larger than $k$ is at least
$$\left(\mfrac{r-1}{r}\right)^{r-1}k^r.$$
Hence,
$$m_k(H)\geq 4k^3/9.$$
This completes the proof of \Cref{cor:chi}.
\end{proof}

\subsection{Organization} The rest of this paper is organized as follows. In \Cref{sec:powers}, we reduce the proof of \Cref{thm:main} to its two main steps, \Cref{lem:shearerapp,lem:indcounting}. We then prove the first lemma in \Cref{sec:shearer}, and the second in \Cref{sec:counting}.

\section{Raising hypergraphs to hypergraph powers}\label{sec:powers}

For sets $A,B$, we denote the set of functions from $A$ to $B$ by $B^A$. If $A=[n]$ for some $n\in \poi$, we write $B^n$ instead of $B^{[n]}$, and identify each $x\in B^n$ with the ordered $n$-tuple $(x(1),\ldots,x(n))$; thus, $B^n$ is identical to the $n$-fold Cartesian product of $B$. For $x\in B^{n}$ and $i\in [n]$, we sometimes write $x_{i}$ instead of $x(i)$ (while making sure this does not cause any ambiguity). For instance, note that for each $t\in \poi$, the $3$-graph $\bt_t$ may equivalently be defined to have vertex set $V(\bt_t)=[3]^t$, and a $3$-subset $\{x,y,z\}$ of $[3]^t$ is a hyperedge of $\bt_t$ if and only if there exists $i\in[t]$ such that $x_j=y_j=z_j$ for all $j\in[i-1]$, and $\{x_i,y_i,z_i\}=\{1,2,3\}$. This is the definition that we will use in the remainder of this paper.
\medskip

Let $k\in\poi$ and let $G=(V,E)$ be a $k$-graph. For every $n\in \poi$, the \textit{$n^{\rm th}$ tensor power} of $G$, denoted $G^n$, is the $k$-graph with
$V(G^n)=V^n$
and
$$E(G^n)=\left\{S\subseteq V^n:|S|=k,\ \{x(i):x\in S\}\in E\text{ for every }i\in[n]\right\}.$$
So, the hyperedges of $G^n$ are exactly those $k$-subsets of $V^n$ whose projection on every ``coordinate'' $i\in [n]$ is a hyperedge of $G$.
\medskip

The key notion in our construction of the $3$-graph $G$ in \Cref{thm:main} is an extension of tensor powers in which a $k$-graph is raised to the power of an arbitrary hypergraph. For $k\in\poi$, a $k$-graph $G=(V_1,E_1)$ and a hypergraph $\Gamma=(V_2,E_2)$, let $G^\Gamma$ be the $k$-graph with $V\left(G^\Gamma\right)=V_1^{V_2}$
and
$$E\left(G^\Gamma\right)=\bigcup_{e\in E_2}\left\{S\subseteq V_1^{V_2}:|S|=k,\ \{x(v):x\in S\}\in E_1\text{ for all }v\in e\right\}.$$
In other words, the hyperedges of $G^\Gamma$ are exactly those $k$-subsets of $V_1^{V_2}$ whose projections on every coordinate in \textit{some} hyperedge $e$ of $\Gamma$ are hyperedges of $G$. For instance, when $\Gamma=([n],\{[n]\})$, then $G^\Gamma=G^n$.
\medskip

The two lemmas below are the main ingredients in the proof of \Cref{thm:main}. We need a few definitions before stating them. For a hypergraph $G$, the \textit{degree} of a vertex $v\in V(G)$ is the number of hyperedges of $G$ containing $v$. For $n\in \poi$, an ordered $n$-tuple $(v_1,\ldots,v_n)$ of not necessarily distinct vertices in $G$ is \textit{stable} if its support $\{v_1,\ldots,v_n\}$ is a stable set. The set of all stable $n$-tuples in $G$ is denoted $\mf T_n(G)$, and we write $\mf t_n(G)=|\mf T_n(G)|$.
\medskip

Our first lemma gives an upper bound on the number of stable tuples in a hypergraph power in terms of its restrictions to the hyperedges of the exponent:

\begin{restatable}{lemma}{shearerapp}\label{lem:shearerapp}
    Let $k,d\in \poi$, let $G=(V_1,E_1)$ be a $k$-graph, and let $\Gamma=(V_2,E_2)$ be a hypergraph with minimum degree at least $d$. Then, for every $n\in \poi$, we have
    $$\mf{t}_n\! \left(G^{\Gamma}\right)\leq \left(\prod_{e\in E_2}\mf{t}_n\! \left(G^{|e|}\right)\right)^{\! 1/d}.$$
\end{restatable}

The second lemma upper-bounds the number of stable tuples in tensor powers of $\bt_t$.

\begin{restatable}{lemma}{indcounting}\label{lem:indcounting}
    For all $n,s\in \poi$ with $s\geq 8$ and $n\geq 3^{(2s)^6}$, there exists $t\in \poi$ such that
    $$\mf{t}_{n}(\bt_t^s)<\left(\mfrac{3^{st}}{n^{1/30}}\right)^{\! \! n}.$$
\end{restatable}

With \Cref{lem:shearerapp,lem:indcounting} in hand, we are ready to prove the main result of this paper:

\main*

\begin{proof}[Proof (assuming \Cref{lem:shearerapp,lem:indcounting})]
Write $\eta=30h(c+1)$ and let $n\geq 3^{\eta^{6h}}$ be a fixed integer.
Let
$s=\binom{\eta-1}{h-1}$. Then $s\geq \eta-1\geq 119$. Also, $2s\leq 2(\eta-1)^{h-1}<\eta^{h}$, and so $n\geq 3^{\eta^{6h}}>3^{(2s)^6}$. Therefore, by \Cref{lem:indcounting},

\sta{\label{st:countingapp} There exists $t\in \poi$ such that $\mf{t}_{n}(\bt_t^s)<\left(\mfrac{3^{st}}{n^{1/30}}\right)^{\! \! n}$.}

Henceforth, let $t\in \poi$ be as in \eqref{st:countingapp}. For each $i\in [\eta]$, define
$e_i=\left\{S\subseteq [\eta]: |S|=h,\ i\in S\right\}$. Let $\Gamma$ be the hypergraph with
$V(\Gamma)=\left\{S\subseteq [\eta]: |S|=h\right\}$ and $E(\Gamma)=\{e_i:i\in [\eta]\}$. Then $\Gamma$ is an $s$-graph with $\eta$ hyperedges in which every vertex has degree $h$.
\medskip

Let
$\Omega=\bt_t^{\Gamma}$. Then $\Omega$ is a $3$-graph with
$|V(\Omega)|=3^{t\binom{\eta}{h}}=3^{st\eta/h}$. We further claim that:

\sta{\label{st:Hfree} Every subgraph $H$ of $\Omega$ with $h$ hyperedges is iteratively tripartite.}

For every $f\in E(H)$, choose $i_f\in[\eta]$ such that $\{x(S):x\in f\}\in E(\bt_t)$ for every $S\in e_{i_f}$. Choose an $h$-subset $S_0$ of $[\eta]$ containing $\{i_f:f\in E(H)\}$ (this is possible since $|E(H)|=h$). Then $S_0\in e_{i_f}$ for every $f\in E(H)$, and so $x\mapsto x(S_0)$ maps every hyperedge of $H$ to a hyperedge of $\bt_t$. Choose $q\in\poi$ with $3^q\geq|V(H)|$, and an injection $x\mapsto \hat{x}$ from $V(H)$ to $[3]^q$. For every $x\in V(H)$, define $x^+\in [3]^{t+q}$ such that
$x^+|_{[t]}=x(S_0)$, and $x^+(t+i)=\hat{x}(i)$ for every $i\in [q]$. So, $x^+$ is obtained by ``concatenating'' $x(S_0)$ and $\hat{x}$ to form an element of $[3]^{t+q}$. Then $x\mapsto x^+$ is injective, and sends every hyperedge of $H$ to a hyperedge of $\bt_{t+q}$. Hence, $H$ is isomorphic to a subgraph of $\bt_{t+q}$. This proves \eqref{st:Hfree}.
\medskip

Now, let $N=n^c$. Note that by \eqref{st:countingapp} and since $\mf t_n(\bt_t^s)\geq 1$, we have $3^{st}>n^{1/30}$, and so
$$|V(\Omega)|=(3^{st})^{\eta/h}>n^{\eta/(30h)}=n^{c+1}>N.$$ We prove that there exists an $N$-subset $Z$ of $V(\Omega)$ such that $\alpha(\Omega[Z])<n$; this, together with \eqref{st:Hfree}, would imply that $G=\Omega[Z]$ is the desired $3$-graph. Suppose for a contradiction that $\alpha(\Omega[Z])\geq n$ for every $N$-subset $Z$ of $V(\Omega)$. Let $\mca{I}$ be the set of all pairs $(Z,S)$ where $Z$ is an $N$-subset of $V(\Omega)$ and $S\subseteq Z$ with $|S|=n$ is a stable set in $\Omega$. It follows that\footnote{Since $\mf{t}_n(\Omega)$ counts stable $n$-tuples, we could save a factor of $n!$ in the upper bound on $|\mca{I}|$. But that would not make a difference in the proof, so we omit it.}
$$\binom{3^{st\eta/h}}{N}\leq |\mca{I}|\leq \mf{t}_n(\Omega)\binom{3^{st\eta/h}-n}{N-n}.$$
We deduce that
$$\mf{t}_n(\Omega)\geq \dfrac{\bigmath{\binom{3^{st\eta/h}}{N}}}{\bigmath{\binom{3^{st\eta/h}-n}{N-n}}}
=\dfrac{\bigmath{\binom{3^{st\eta/h}}{n}}}{\bigmath{\binom{N}{n}}}\geq \dfrac{\left(\mfrac{3^{st\eta/h}}{n}\right)^{\! n}}{N^n}
=\left(\mfrac{3^{st}}{n^{\frac{(c+1)h}{\eta}}}\right)^{\! n\eta/h}
=\left(\mfrac{3^{st}}{n^{1/30}}\right)^{\! n\eta/h}.$$
Furthermore, since $\Gamma$ is an $s$-graph with $\eta$ hyperedges in which every vertex has degree exactly $h$, by \Cref{lem:shearerapp} applied to $\bt_t$ and $\Gamma$, we have
$\mf{t}_n(\Omega)=\mf{t}_n\!\left(\bt_t^{\Gamma}\right)\leq \mf{t}_n(\bt_t^s)^{\eta/h}$. But now
$\mf{t}_n(\bt_t^s)\geq \left(3^{st}/n^{1/30}\right)^{\! n}$,
contrary to \eqref{st:countingapp}. This completes the proof of \Cref{thm:main}.
\end{proof}

It remains to prove \Cref{lem:shearerapp,lem:indcounting}, which we will do in \Cref{sec:shearer,sec:counting}.

\section{Using Shearer's inequality}\label{sec:shearer}

Here, we prove \Cref{lem:shearerapp}. The proof relies on the ``counting form'' of Shearer's inequality (see \cite{shearer} for further discussion):

\begin{lemma}[See {\cite[Corollary~3.2]{shearer}}]\label{lem:shearer}
    Let $\Gamma=(V,E)$ be a hypergraph with minimum degree at least $d\in \poi$. Then, for every finite set $A$ and every $\mca{X}\subseteq A^{V}$, we have
    $$|\mca{X}|\leq \left(\prod_{e\in E} \big|\{x\big|_{e}:x\in \mca{X}\}\big|\right)^{\! 1/d}.$$
\end{lemma}

Let us now prove \Cref{lem:shearerapp}, which we restate:

\shearerapp*

\begin{proof}
Let $\mca{V}=V\big(G^{\Gamma}\big)=V_1^{V_2}$. Let $A=V_1^{n}$. For each $x\in \mca{V}^{n}$, define $\hat{x}\in A^{V_2}$ with the following rule: for every $v\in V_2$, we have
$$\hat{x}(v)=(x_1(v),\ldots,x_n(v))\in A.$$
Then the map $x\mapsto \hat{x}$ is bijective from $\mca{V}^{n}$ to $A^{V_2}$. Let
$\mca{X}=\left\{\hat{x}:x\in \mf{T}_n\big(G^{\Gamma}\big)\right\}$, and for every $e\in E_2$, let
$\mca{X}_e=\left\{\hat{x}\big|_e:x\in \mf{T}_n\big(G^{\Gamma}\big)\right\}$. It follows that

\sta{\label{st:Xtranspose}$\mca{X}\subseteq A^{V_2}$ with $|\mca{X}|=|\mf{T}_n\big(G^{\Gamma}\big)|=\mf{t}_n\! \left(G^{\Gamma}\right)$.}

Moreover, we show that:

\sta{\label{st:tuplerestriction} $|\mca{X}_e|\leq \mf{t}_n(G^{|e|})$ for all $e\in E_2$.}

Let $e=\{v_{1},\ldots,v_{m}\}$. For each $y\in \mca{X}_{e}\subseteq (V_1^{n})^{e}$, write $y_{j}=y(v_{j})$ for every $j\in [m]$, and define $\check{y}\in (V_1^{m})^{n}$ by 
$$\check{y}(i)=(y_{1}(i),\ldots,y_{m}(i))\in V_1^{m}$$
for every $i\in [n]$. Clearly, the map $y\mapsto \check{y}$ is well defined and injective from $\mca{X}_{e}$ into $(V_1^{m})^{n}$. We further claim that $\check{y}\in \mf{T}_{n}\big(G^m\big)$ for every $y\in \mca{X}_{e}$. Suppose that $\check{y}\notin \mf{T}_{n}\big(G^m\big)$. Then there exists $I\subseteq [n]$ with $|I|=k$ such that $\{\check{y}(i):i\in I\}\in E(G^m)$. Let $x\in \mf{T}_{n}\big(G^{\Gamma}\big)$ such that $y=\hat{x}|_{e}$. Then, by the definition of $\hat{x}$, we have 
$$\check{y}(i)=(x_{i}(v_{1}),\ldots,x_{i}(v_{m})).$$
Since $\{\check{y}(i):i\in I\}\in E(G^m)$, it follows that $\{x_{i}(v_{j}):i\in I\}\in E_1$ for every $j\in [m]$. But now from $e\in E_2$, we deduce that $\{x_{i}:i\in I\}\in E(G^{\Gamma})$, contrary to the choice of $x\in \mf{T}_{n}\big(G^{\Gamma}\big)$.
The claim follows, and so $|\mca{X}_{e}|\leq |\mf{T}_{n}\big(G^m\big)|=\mf{t}_{n}\big(G^m\big)$. This proves \eqref{st:tuplerestriction}.
\medskip

From \eqref{st:Xtranspose}, \eqref{st:tuplerestriction}, and \Cref{lem:shearer}, we deduce that
$$\mf{t}_n\! \left(G^{\Gamma}\right)=|\mca{X}|\leq \left(\prod_{e\in E_2} |\mca{X}_e|\right)^{\! 1/d}\leq \left(\prod_{e\in E_2}\mf{t}_n\! \left(G^{|e|}\right)\right)^{\! 1/d}.$$
This completes the proof of \Cref{lem:shearerapp}.
\end{proof}

\section{Counting stable tuples via cap-set triples}\label{sec:counting}

In this last section, we prove \Cref{lem:indcounting}, starting with a definition. For $M\in\poi$, a \textit{cap-set triple in $[3]^M$} is an ordered triple $(x,y,z)$ such that for every $i\in[M]$, either $x(i)=y(i)=z(i)$ or $\{x(i),y(i),z(i)\}=\{1,2,3\}$. We denote the set of these triples by $\mca C_M$. In particular, $(x,x,x)\in\mca C_M$ for every $x\in[3]^M$. (Identifying $[3]$ with $\mathbb{F}_3$, these are precisely the triples $(x,y,z)\in (\mathbb{F}_3^M)^3$ satisfying $x+y+z=0$, but we do not need this algebraic interpretation.) We say that cap-set triples $(x_1,y_1,z_1),\ldots,(x_m,y_m,z_m)$ are \textit{disjoint} if $x_1,\ldots, x_m$ are pairwise distinct, $y_1,\ldots, y_m$ are pairwise distinct, and $z_1,\ldots, z_m$ are pairwise distinct. (Note that the three sets of entries $\{x_1,\ldots, x_m\}, \{y_1,\ldots, y_m\}$ and $\{z_1,\ldots, z_m\}$ need not be disjoint from one another.)
\medskip

We use the following consequence of the ``arithmetic triangle removal theorem'' of Fox and Lov\'asz \cite{FoxLovasz}. (Their result gives an exponent smaller than $14$.)

\begin{theorem}[Fox and Lov\'asz {\cite[Theorem~3 for $p=3$]{FoxLovasz}}]\label{thm:FoxLovasz} Let $m,M\in \poi$ and let $\varepsilon=m/3^{M}$. Let $(x_1,y_1,z_1),\ldots, (x_m,y_m,z_m)$ be $m$ pairwise disjoint cap-set triples in $[3]^M$. Then there are at least $\varepsilon^{14} 3^{2M}$ cap-set triples in $\{x_1,\ldots, x_m\}\times \{y_1,\ldots, y_m\}\times \{z_1,\ldots, z_m\}$.
\end{theorem}

\begin{corollary}\label{cor:FoxLovasz}
   For all $M\in \poi$ and $X\subseteq [3]^M$, $X^3$ contains at least $|X|^{14}/3^{12M}$ cap-set triples.
\end{corollary}

\begin{proof}
   We may assume that $X\neq \varnothing$. Let $|X|=m\in \poi$; say, $X=\{x_1,\ldots,x_m\}$. Then $(x_1,x_1,x_1),\ldots, (x_m,x_m,x_m)$ are $m$ pairwise disjoint cap-set triples in $[3]^M$; thus, by \Cref{thm:FoxLovasz},
   $X^3$ contains at least $(m/3^{M})^{14}\cdot 3^{2M}=m^{14}/3^{12M}$ cap-set triples.
\end{proof}

The key technical step in the proof of \Cref{lem:indcounting} is the following. It shows that for all $s,t\in \poi$, every suitably dense set of vertices in $\bt_t^s$ contains three disjoint subsets of polynomially smaller density such that every transversal of them is a hyperedge of $\bt_t^s$.

\begin{lemma}\label{lem:zerosumtriples}
   Let $s,t\in \poi$ and let $d\in \mathbb R$ with
 $2(2t)^s\leq d\leq \left(3^t/(2s)\right)^{\! 1/14}$. Then, for every $X\subseteq V(\bt_t^s)$ with $|X|\geq 3^{st}/d$, there exist $X_1,X_2,X_3\subseteq X$ with $|X_1|,|X_2|,|X_3|\geq 3^{st}/d^{15}$ such that for every $(x^{(1)},x^{(2)},x^{(3)})\in X_1\times X_2\times X_3$, we have
$\{x^{(1)},x^{(2)},x^{(3)}\}\in E(\bt_t^s)$.
In particular, $X_1,X_2,X_3$ are pairwise disjoint.
\end{lemma}

\begin{proof}
Let $\mca{T}=\mca{T}_1\setminus \mca{T}_2$, where
$$\mca{T}_1=\{(x,y,z)\in X^3: (x_i,y_i,z_i)\in \mca{C}_t \text{ for all }i\in [s]\};$$
$$\mca{T}_2=\{(x,y,z)\in \mca{T}_1: x_i=y_i=z_i \text{ for some }i\in [s]\}.$$
Our first goal is to give a lower bound on $|\mca{T}|$:

\sta{\label{st:lowertriple} We have $|\mca{T}|\geq 3^{2st}/(2d^{14})$.}

For $x\in \left([3]^{t}\right)^{\! s}$, define $\hat{x}\in [3]^{st}$ with $\hat{x}((i-1)t+j)=x_i(j)$ for all $i\in [s]$ and $j\in [t]$; in other words, $\hat{x}$ is obtained by ``concatenating'' the $t$-tuples $x_1,\ldots, x_s$ in order. Then $x\mapsto \hat{x}$ is bijective from $\left([3]^{t}\right)^{\! s}$ to $[3]^{st}$. Let $\hat{X}=\{\hat{x}:x\in X\}\subseteq [3]^{st}$. Then $(x,y,z)\mapsto (\hat{x},\hat{y},\hat{z})$ is bijective from $\mca{T}_1$ to $\mca{C}_{st}\cap \hat{X}^3$. Since $|\hat{X}|=|X|\geq 3^{st}/d$, it follows from \Cref{cor:FoxLovasz} that $|\mca{C}_{st}\cap \hat{X}^3|\geq 3^{2st}/d^{14}$, and so
$|\mca{T}_1|\geq 3^{2st}/d^{14}$.
Moreover, to construct a triple $(x,y,z)\in \mca{T}_2$, there are $s$ choices of $i\in [s]$ for which $x_i=y_i=z_i$, and at most $3^{st}$ choices of $x$. Once $i,x$ are chosen, there are at most $3^{st-t}$ choices of $y$ (because $x_i=y_i$), and once $y$ is also chosen, there is at most one choice of $z$ (as $(x,y,z)\in \mca{T}_1$). We deduce that $|\mca{T}_2|\leq s3^{2st-t}$, which combined with $d^{14}\leq 3^{t}/(2s)$ yields
$|\mca{T}_2|\leq 3^{2st}/(2d^{14})$.
This proves \eqref{st:lowertriple}.
\medskip

Next, we classify the triples in $X^3$ that form a hyperedge of $\bt_t^s[X]$ by the ``hyperedge certificate'' prefixes of each of their $s$ coordinates. For $\ell\in [t]^s$, we denote by $\mca{P}_{\ell}$ the set of all $3s$-tuples
$p=(p_{1,1},\ldots,p_{1,s},p_{2,1},\ldots,p_{2,s},p_{3,1},\ldots,p_{3,s})$
such that for every $i\in [s]$,
\vspace{1mm}

\begin{itemize}[itemsep=1mm]
    \item $p_{1,i},p_{2,i},p_{3,i}\in [3]^{\ell_i}$;
    \item $p_{1,i}(j)=p_{2,i}(j)=p_{3,i}(j)$ for all $j\in [\ell_i-1]$; and
    \item   $\{p_{1,i}(\ell_i),p_{2,i}(\ell_i),p_{3,i}(\ell_i)\}=\{1,2,3\}$.
\end{itemize}
\vspace{1mm}

We also write
$\mca{P}=\bigcup_{\ell\in [t]^s}\mca{P}_{\ell}$. Let $\ell\in [t]^{s}$ and let $p\in \mca{P}_{\ell}$. Then, for every $k\in [3]$, we define
$$X_{p,k}=\{x\in X:x_i(j)=p_{k,i}(j)\text{ for all }i\in [s]\text{ and }j\in [\ell_i]\}.$$

It is straightforward to deduce from these definitions that:

\sta{\label{st:prefixcount} The following hold:
\begin{itemize}
\item For every $\ell\in [t]^s$, we have $|\mca{P}_{\ell}|=2^s 3^{\ell_1+\cdots+\ell_s}$.
\item For every triple $(x,y,z)\in \mca{T}$, there exist unique $\ell\in [t]^{s}$ and $p\in \mca{P}_{\ell}$ such that $(x,y,z)\in X_{p,1}\times X_{p,2}\times X_{p,3}$. In particular,
    $$|\mca{T}|=\sum_{\ell\in [t]^s}\sum_{p\in \mca{P}_{\ell}}|\mca{T}\cap (X_{p,1}\times X_{p,2}\times X_{p,3})|.$$
\item For each $p\in \mca{P}$ and every triple $(x^{(1)},x^{(2)},x^{(3)})\in X_{p,1}\times X_{p,2}\times X_{p,3}$, we have
$\{x^{(1)},x^{(2)},x^{(3)}\}\in E(\bt_t^s)$.
In particular, $X_{p,1},X_{p,2},X_{p,3}$ are pairwise disjoint.
\end{itemize}}

We further claim that:

\sta{\label{st:uppertriple} There exists $p_0\in \mca{P}$ for which
$|X_{p_0,1}|,|X_{p_0,2}|,|X_{p_0,3}|\geq 3^{st}/d^{15}$.}

For each $p\in \mca{P}$, let
$m_p=\min\{|X_{p,1}|,|X_{p,2}|,|X_{p,3}|\}$.
Fix $\ell\in [t]^{s}$ and $p\in \mca{P}_{\ell}$. Assume without loss of generality that $|X_{p,1}|=m_p$. For a triple $(x,y,z)\in \mca{T}\cap (X_{p,1}\times X_{p,2}\times X_{p,3})$, there are $m_p$ choices of $x\in X_{p,1}$. Once $x$ is chosen, there are $|X_{p,2}|\leq 3^{st-(\ell_1+\cdots+\ell_s)}$ choices of $y\in X_{p,2}$, and once $y$ is also chosen, there is at most one choice of $z\in X_{p,3}$ (because $(x,y,z)\in \mca{T}_1$). It follows that
$$|\mca{T}\cap (X_{p,1}\times X_{p,2}\times X_{p,3})|\leq m_p\,3^{st-(\ell_1+\cdots+\ell_s)}.$$
Choose $p_0\in \mca{P}$ with $m_{p_0}$ as large as possible. By the first two bullets in \eqref{st:prefixcount}, we have
\begin{align*}
    |\mca{T}|&=\sum_{\ell\in [t]^s}\sum_{p\in \mca{P}_{\ell}}|\mca{T}\cap (X_{p,1}\times X_{p,2}\times X_{p,3})|\\
    &\leq \sum_{\ell\in [t]^s}\sum_{p\in \mca{P}_{\ell}}3^{st-(\ell_1+\cdots+\ell_s)}\cdot m_p\\
    &\leq \sum_{\ell\in [t]^s}|\mca{P}_{\ell}|\cdot 3^{st-(\ell_1+\cdots+\ell_s)}\cdot m_{p_0}\\
    &\leq (2t)^s 3^{st}\cdot m_{p_0}.
\end{align*}
Combining this with \eqref{st:lowertriple} and $d\geq 2(2t)^s$, we deduce that
$$m_{p_0}\geq \mfrac{3^{st}}{2(2t)^sd^{14}}\geq \mfrac{3^{st}}{d^{15}}.$$
This proves \eqref{st:uppertriple}.
\medskip

Now, let $p_0\in \mca{P}$ be as in \eqref{st:uppertriple}. Then $X_{p_0,1},X_{p_0,2},X_{p_0,3}\subseteq X$ each have cardinality at least $3^{st}/d^{15}$, and by the third bullet in \eqref{st:prefixcount}, for every triple $(x^{(1)},x^{(2)},x^{(3)})\in X_{p_0,1}\times X_{p_0,2}\times X_{p_0,3}$, we have $\{x^{(1)},x^{(2)},x^{(3)}\}\in E(\bt_t^s)$; in particular, $X_{p_0,1},X_{p_0,2},X_{p_0,3}$ are pairwise disjoint. This completes the proof of \Cref{lem:zerosumtriples}.
\end{proof}

Finally, we use \Cref{lem:zerosumtriples} to prove \Cref{lem:indcounting}, which we restate:

\indcounting*

\begin{proof}
Let $k\geq 6$ be the (unique) integer for which
$$3^{(2s)^k}\leq n<3^{(2s)^{k+1}}.$$
We prove that
$$t=(2s)^{k+2}$$
satisfies the lemma. Let 
$$d=n^{1/15}\geq 1.$$
We first show that:

\sta{\label{st:dbounds} We have $2(2t)^s\leq d<\left(3^t/(2s)\right)^{\! 1/14}$.}

For the upper bound, since $n<3^{(2s)^{k+1}}$, it follows that
$$2sd^{14}\leq 2sd^{15}=2sn<2s\cdot 3^{(2s)^{k+1}}
\leq 3^{2s}\cdot 3^{(2s)^{k+1}}\leq 3^{2(2s)^{k+1}}\leq 3^{(2s)^{k+2}}=3^t.$$
For the lower bound, since $n\geq 3^{(2s)^k}$ and $s\geq 8$, it follows that
$d=n^{1/15}>3^{(2s)^{k-1}}$.
Also, note that $3^{m}\geq m^{2}$ for every integer $m\geq 2$. Thus, since $k\geq 6$ and $t>4$, we have
$$3^{(2s)^{k-1}}=\left(3^{(2s)^{k-2}}\right)^{\! 2s}\geq \left((2s)^{2k-4}\right)^{\! 2s}
\geq \left((2s)^{k+2}\right)^{\! 2s}=t^{2s}\geq (4t)^s\geq 2(2t)^s.$$
This proves \eqref{st:dbounds}.
\medskip

The following inductive argument is at the heart of the proof of \Cref{lem:indcounting}. It obtains, for every $X\subseteq V(\bt_t^s)$ of controlled density, a small family of small ``containers'' for all stable sets in $\bt_t^s[X]$.

\sta{\label{st:ternaryinduction} Let $\nu\in \poi$ and let $X\subseteq V(\bt_t^s)$ with $|X|\leq (\nu 3^{st})/d^{15}$. Then there is a set $\mca{S}$ of at most $3^{\nu-1}$ subsets of $X$, each of cardinality at most $3^{st}/d$, such that every stable set in $\bt_t^s[X]$ is a subset of one of the sets in $\mca{S}$.}

We proceed by induction on $\nu$. If $|X|\leq 3^{st}/d$, then \eqref{st:ternaryinduction} follows by taking $\mca{S}=\{X\}$. So, we may assume that $|X|>3^{st}/d$; in particular, $|X|>3^{st}/d^{15}$ (because $d\geq 1$), and thus $\nu\geq 2$. Since $|X|\geq 3^{st}/d$, by \Cref{lem:zerosumtriples}, there exist $X_1,X_2,X_3\subseteq X$ with $|X_1|,|X_2|,|X_3|\geq 3^{st}/d^{15}$ such that for every $(x^{(1)},x^{(2)},x^{(3)})\in X_1\times X_2\times X_3$, we have $\{x^{(1)},x^{(2)},x^{(3)}\}\in E(\bt_t^s)$. In particular, every stable set in $\bt_t^s[X]$ is contained in one of $X\setminus X_1,X\setminus X_2,X\setminus X_3$. For each $i\in [3]$, since $|X_i|\geq 3^{st}/d^{15}$ and $|X|\leq (\nu 3^{st})/d^{15}$, it follows that $|X\setminus X_i|\leq ((\nu-1)3^{st})/d^{15}$. Therefore, by the inductive hypothesis, for every $i\in [3]$, there is a set $\mca{S}_i$ of at most $3^{\nu-2}$ subsets of $X\setminus X_i$, each of cardinality at most $3^{st}/d$, such that every stable set in $\bt_t^s[X\setminus X_i]$ is a subset of one of the sets in $\mca{S}_i$. Now the result follows by taking $\mca{S}=\mca{S}_1\cup \mca{S}_2\cup \mca{S}_3$. This proves \eqref{st:ternaryinduction}.
\medskip

Let $\nu=n=d^{15}$. Then $|V(\bt_t^s)|=3^{st}=(\nu 3^{st})/d^{15}$. By \eqref{st:ternaryinduction}, there is a set $\mca{S}$ of at most $3^{\nu-1}<3^{n}$ subsets of $V(\bt_t^s)$, each of cardinality at most $3^{st}/d$, such that every stable set in $\bt_t^s$ is a subset of one of the sets in $\mca{S}$. In particular, every stable $n$-tuple of vertices in $\bt_t^s$ is a stable $n$-tuple of vertices in $\bt_t^s[S]$ for some $S\in \mca{S}$. Consequently,
$$\mf{t}_n(\bt_t^s)\leq \sum_{S\in \mca{S}}\mf{t}_n(\bt_t^s[S])\leq \sum_{S\in \mca{S}}|S|^n< 3^{n}\cdot \left(\dfrac{3^{st}}{d}\right)^n=\left(\dfrac{3^{st+1}}{d}\right)^n.$$
Furthermore, the lower bound on $d$ from \eqref{st:dbounds} yields in particular that $d>9$. Hence,
$$\mf{t}_n(\bt_t^s)\leq\left(\dfrac{3^{st+1}}{d}\right)^n<\left(\dfrac{3^{st}}{d^{1/2}}\right)^n=\left(\dfrac{3^{st}}{n^{1/30}}\right)^n.$$
This completes the proof of \Cref{lem:indcounting}.
\end{proof}

\section{AI declaration}
All proofs and writing are entirely due to the authors. ChatGPT was used to proofread the final draft.

\bibliographystyle{abbrv}
\bibliography{ref}
\end{document}